\documentclass[a4paper,11p]{amsart}

\usepackage{amsfonts,amsthm,latexsym,amsmath,amssymb,amscd,amsmath, mathrsfs, epsf}
\usepackage[utf8]{inputenc}

\usepackage[dvipsnames]{xcolor}
\usepackage{tikz}
\usetikzlibrary{arrows,matrix}
\usepackage{blkarray}
\usepackage[linewidth=1pt]{mdframed}
\usepackage{mathtools}

\usepackage{algorithm}
\usepackage{algorithmicx,algpseudocode}
\algrenewcommand{\algorithmiccomment}[1]{\hspace{-19pt} {\normalsize\color{black}$\triangleright\;$\textit{#1}}}

\usepackage{enumitem}

\usepackage{hyperref}
\hypersetup{
    colorlinks,
    linkcolor={MidnightBlue},
    citecolor={Green}
}
\usepackage[nameinlink]{cleveref}

\usepackage[margin=1.in]{geometry}
\usepackage{marginnote}
\usepackage[parfill]{parskip}
\begingroup
\makeatletter
   \@for\theoremstyle:=definition,remark,plain\do{%
     \expandafter\g@addto@macro\csname th@\theoremstyle\endcsname{%
        \addtolength\thm@preskip\parskip
     }%
   }
\endgroup

\newtheoremstyle{indented}{5pt}{5pt}{\itshape}{2.5em}{\bfseries}{.}{.5em}{}

\theoremstyle{plain}
\newtheorem{theorem}{Theorem}[section]
\newtheorem{lemma}[theorem]{Lemma}
\newtheorem{proposition}[theorem]{Proposition}
\newtheorem{corollary}[theorem]{Corollary}
\newtheorem{definition}[theorem]{Definition}

\newtheorem{question}{Question}
\newtheorem{claim}{Claim}[section]

\theoremstyle{definition}
\newtheorem{example}[theorem]{Example}
\newtheorem{remark}[theorem]{Remark}

\newtheorem{alg}{Algorithm}

\theoremstyle{indented}
\newtheorem*{claim*}{\indent Claim}

\newcommand{\exampleend}{ \hfill$\spadesuit$}

\renewcommand{\a}{\alpha}

\renewcommand{\a}{\alpha}

\newcommand{\se}{\subseteq}
\newcommand{\len}{{\rm len}}
\newcommand{\reg}{{\rm reg}}
\newcommand{\Res}{{\rm Res}}

\newcommand{\tn}{\textnormal}

\newcommand{\Ann}{\tn{Ann}}

\newcommand{\N}{\mathbb{N}}

\newcommand{\Pn}{\bbP^n}

\newcommand{\codim}{\textnormal{codim}}

\newcommand{\calD}{\mathcal{D}}

\newcommand{\calR}{\mathcal{R}}
\newcommand{\calRa}{\underline{\mathcal{R}}}
\newcommand{\calS}{\mathcal{S}}
\newcommand{\calSa}{\underline{\mathcal{S}}}

\newcommand{\bbN}{\mathbb{N}}

\newcommand{\bbP}{\mathbb{P}}

\newcommand{\HF}{\mathrm{HF}}

\newcommand{\Sym}{\mathrm{Sym}}
\newcommand{\contract}{\mathbin{\mathpalette\dointprod\relax}}
\newcommand{\dointprod}[2]{%
  \raisebox{\depth}{\scalebox{1}[-1]{$#1\lnot$}}}

\newcommand{\todo}[1]{\marginnote{{\tiny\color{red}#1}}}

\title{On schemes evinced by generalized additive decompositions and their regularity}
\author{Alessandra Bernardi, Alessandro Oneto, Daniele Taufer*}

\address[A. Bernardi, A. Oneto]{Universit\`a di Trento, Via Sommarive, 14 - 38123 Povo (Trento), Italy}
\email{alessandra.bernardi@unitn.it, alessandro.oneto@unitn.it}

\address[D. Taufer]{KU Leuven - Celestijnenlaan 200A, B-3001 Leuven (Belgium)}
\email{daniele.taufer@kuleuven.be (*corresponding author)}

\makeatletter
\@namedef{subjclassname@2020}{\textup{2020} Mathematics Subject Classification}
\makeatother


\begin{document}
\newenvironment{AB}{\color{RubineRed}}{\color{black}}
\newenvironment{DT}{\color{olive}}{\color{black}}


\begin{abstract}
    We define and explicitly construct schemes evinced by generalized additive decompositions (GADs) of a given $d$-homogeneous polynomial $F$.
    We employ GADs to investigate the regularity of $0$-dimensional schemes apolar to $F$, focusing on those satisfying some minimality conditions.
    We show that irredundant schemes to $F$ need not be $d$-regular, unless they are evinced by special GADs of $F$.
    Instead, we prove that tangential decompositions of minimal length are always $d$-regular, as well as irredundant apolar schemes of length at most $2d+1$.
\end{abstract}

\maketitle

\section{Introduction}


Algebraic and geometric properties of $0$-dimensional schemes have been largely studied from several perspectives in algebraic geometry, commutative algebra, and computational algebra. Through \emph{apolarity theory}, these studies find applications in the study of \emph{additive decompositions} of homogeneous polynomials and, more in general, \emph{tensor decompositions} \cite{landsberg2012tensors,bernardi2018hitchhiker,BC:book}.

In this paper, we are interested in $0$-dimensional schemes that are \emph{apolar} to a given $d$-homogeneous polynomial $F$, namely the $0$-dimensional schemes defined by ideals annihilating $F$ by derivation.
Understanding the possible Hilbert functions of {\it minimal} apolar schemes 
is a deep and largely open question, which could give useful information on the nature of additive decompositions of polynomials and {\it secant varieties}, and whose grasp is challenging even in moderately small cases \cite{RS00,iliev2001,curvilinear,elias2012isomorphism,carlini2012solution,landsberg2013equations,buczynska2013waring,BB14:Secant,CJN,jelisiejew2018vsps,BJMK18:Polynomials,Chiantini:quantum,MO,buczynska2021apolarity,angelini2023waring}.

Our work aims to study when these Hilbert functions stabilize, and more specifically at discerning essential conditions for a given $d$-homogeneous polynomial to have \emph{minimal} $0$-dimensional apolar schemes that are regular in degree $d$.
This subtle problem carries far-reaching implications spanning the domains of classical algebraic geometry and complexity theory. In the context of algebraic geometry, these concepts are part of a longstanding tradition of exploring secant varieties and Waring problems, see \cite{bernardi2018hitchhiker} for a general overview.
From a complexity theory perspective, the knowledge of the regularity of minimal apolar schemes to a given polynomial might improve the efficiency of symbolic algorithms for computing ranks and minimal decomposition of polynomials \cite{BRACHAT20101851,BT20:Waring,BDHM}. 

\subsection{Additive decompositions}

As already recalled, the study of apolar schemes is related to notions of \emph{rank} and \emph{additive decompositions} associated with homogeneous polynomials.
The minimal length of a $0$-dimensional scheme apolar to $F$ is the \emph{cactus rank} of $F$ \cite{RS:11,BB14:Secant}. 
If we restrict to schemes that are locally contained in $(d+1)$-fat points, then they correspond to \emph{generalized additive decompositions} (GADs) of $F$, namely expressions as
\[F = \sum_{i=1}^r L_i^{d-k_i}G_i,\]
where the $L_i$'s are pairwise non-proportional linear forms not dividing the corresponding $G_i$'s \cite{IaKa:book,BBM14:comparison,BT20:Waring}.
Special cases of such decompositions include \emph{tangential decompositions}, when $k_i = 1$ \cite{BT20:Waring,CGG, ballico:tg}, and \emph{Waring decompositions}, when $k_i = 0$ \cite{Ger96,bernardi2018hitchhiker}.

This algebraic description of ranks and additive decompositions has a geometric interpretation in terms of \emph{Veronese varieties} and their \emph{secant varieties} \cite{Zak:tgsec, A:joins, bernardi2018hitchhiker}.
A Waring decomposition corresponds to a set of points on the Veronese variety whose linear span contains the projective point corresponding to the polynomial $F$.
Analogously, tangential decompositions (generalized additive decompositions, respectively) correspond to a set of points on the tangential variety (osculating varieties, respectively) of the Veronese variety whose linear span contains the projective point of $F$ \cite{bernardi2018hitchhiker, CGG, BCGI07:Osculating, Bernardi2009982, BF}.
In this view, GADs parameterize generic points of a \emph{joint variety} of osculating varieties to a certain Veronese variety.

\subsection{Content of the paper and main results} 

After recalling the standard definition and results in \Cref{sec:Preliminaries}, we define and provide an explicit construction of schemes evinced by GADs in \Cref{sec:GAD}.
This construction locally agrees with the natural apolar schemes defined in \cite{BJMK18:Polynomials}, but is made effective by delving into the computational details.
An implementation of this construction routine in Macaulay2 \cite{M2} and Magma \cite{MR1484478} can be found in \cite{Repo}.

In \Cref{sec:irred} we investigate the weaker and more geometric irredundancy condition, i.e. we look at schemes that are minimal by inclusion among the apolar schemes to a given form $F$ of degree $d$.
With \Cref{ex:irredbutnotmin} we observe that schemes evinced by GADs might well be redundant, whereas we prove in \Cref{prop:ContainsGAD} that irredundant schemes are evinced by a GAD of $F$ precisely when their connected components are contained in $(d+1)$-fat points.
Therefore, all schemes apolar to $F$ with \emph{short} components are evinced by certain families of GADs of $F$.
However, \Cref{ex:GAD_nonContained} shows that schemes with \emph{long} components may only arise from GADs of higher degree polynomials.

In \Cref{sec:regularity} we tackle the regularity of minimal apolar schemes.
We show that non-redundancy to a degree-$d$ form is not enough to ensure $d$-regularity.
Indeed, in \Cref{irregular:irredundant,ex:perazzo} we present degree-$d$ homogenous polynomials admitting an apolar scheme that is irredundant but not $d$-regular.
However, we notice that in both cases such schemes are not minimal by length.

In \Cref{prop:externalconditions} we show that the addenda constituting a GAD evincing an irredundant scheme $Z$ may never appear in its inverse systems.
We use this result in \Cref{prop:lowmultiplicity} to guarantee $d$-regularity for schemes evinced by GADs such that the $L_i$'s are linearly independent and the $k_i$'s are small enough, regardless of the scheme being minimal.
However, we point out in \Cref{rmk:sharp} that all the assumptions of \Cref{prop:lowmultiplicity} are sharp.

Drawing from the intuition that schemes with components of low multiplicity usually exhibit low regularity, in \Cref{prop:2jets} we prove that minimal tangential decompositions of degree-$d$ forms always evince $d$-regular schemes.
\Cref{ex:perazzo} shows the condition of having minimal length is essential, while irredundancy is not enough.
This answers an open question from \cite[Remark 5.4]{BT20:Waring}.

Finally, we show in \Cref{prop:shortschemes} that if the cactus rank of a degree-$d$ form is not greater than $2d+1$, then non-redundancy is actually enough to guarantee $d$-regularity. 
In particular, all the schemes of minimal length apolar to degree-$d$ forms with length smaller or equal to $2d+1$ are $d$-regular.

\subsection*{Acknowledgements} We sincerely thank E. Ballico, W. Buczy\'nska, J. Buczy\'nski, M.V. Catalisano, C. Ciliberto, J. Jelisiejew and B.~Mourrain for fruitful conversations. 
We also thank the anonymous reviewer for the careful proofreading.
DT acknowledges the hospitality of the TensorDec Laboratory during a research stay at the Department of Mathematics at the University of Trento, where part of the present work has been conducted.

\subsection*{Funding}
DT has been supported by the European Union's H2020 Programme ERC-669891, and by the Research Foundation - Flanders via the FWO postdoctoral fellowship 12ZZC23N and the travel grant V425623N. 
AB has been partially supported by GNSAGA of INDAM. AB and AO have been funded by the European Union under NextGenerationEU. PRIN 2022 Prot. n. 2022ZRRL4C$\_$004 and Prot. n. 20223B5S8L, respectively. Views and opinions expressed are however those of the author(s) only and do not necessarily reflect those of the European Union or European Commission.  Neither the European Union nor the granting authority can be held responsible for them. 
\begin{center}\includegraphics[height=1.95cm]{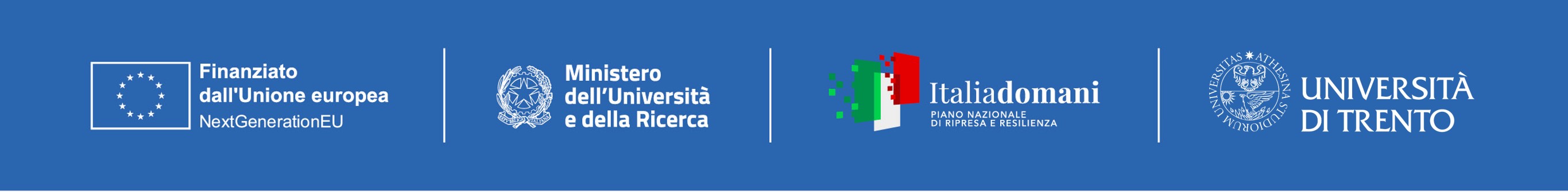}\end{center}

All the authors have been partially supported by the Thematic Research Programme ``Tensors: geometry, complexity and quantum entanglement", University of Warsaw, Excellence Initiative – Research University and the Simons Foundation Award No. 663281.


\section{Preliminaries}\label{sec:Preliminaries}


In this paper, $\Bbbk$ will always be an algebraically closed field of characteristic $0$. 
Given $\alpha = (\alpha_0,\ldots,\alpha_n)$ and $\beta = (\beta_0,\ldots,\beta_n)$ in $\bbN^{n+1}$, let $|\alpha| = \sum_{i=0}^n \alpha_i$ and $\alpha! = \prod_{i=0}^n \alpha_i!$. We write $\alpha \succeq \beta$ if $\alpha_i \geq \beta_i$ for every $0 \leq i \leq n$. We use the standard short notation $X^\alpha = X_0^{\alpha_0}\cdots X_n^{\alpha_n}$.

\subsection{Apolarity} \label{sec:apolarity}
Let $\calS = \Bbbk[X_0,\ldots,X_n] = \bigoplus_{d \in \N} \calS_d$ and $\calR = \Bbbk[Y_0,\ldots,Y_n] = \bigoplus_{d \in \N} \calR_d$ be standard graded polynomial rings, where $\calS_d$ and $\calR_d$ denote the $\Bbbk$-vector spaces of degree-$d$ homogeneous polynomials. 
We also write $\calS_{\leq d} = \bigoplus_{e \leq d} \calS_e$ and $\calR_{\leq d} = \bigoplus_{e \leq d} \calR_e$.

We consider the apolarity action of $\calR$ on $\calS$ given by differentiation, i.e.,
\begin{equation}\label{action:diff}
Y^\beta \circ X^{\alpha} = 
	\begin{cases}
		\partial_\beta(X^\alpha) = \frac{\alpha!}{(\alpha-\beta)!}X^{\alpha-\beta} & \text{ if } \alpha \succeq \beta, \\
		0 & \text{ otherwise},
	\end{cases}
\end{equation}
extended by $\Bbbk$-linearity. Given $F \in \calS$, we consider its annihilator 
    \[
        \Ann(F) = \{G \in \calR ~:~ G \circ F = 0\}, 
    \]
which is an ideal of $\calR$.

This action defines a non-degenerate perfect pairing $\calR_d \times \calS_d \to \Bbbk$ for every $d \in \bbN$.
Given a subspace $V \se \calS_d$, we denote by $V^\perp \se \calR_d$ its orthogonal space with respect to such pairing.
If $V = \langle F \rangle$, we simply denote its orthogonal space by $F^{\perp}$. 

\begin{remark}
     A classical result by Macaulay \cite{macaulay1916algebraic} shows that graded Artinian Gorenstein algebras are all, and only, quotient rings of polynomial rings by annihilator ideals of homogeneous polynomials, see \cite[Theorem 8.7]{Ger96}, \cite[Lemma 2.12]{IaKa:book} or \cite[Theorem 21.6]{eisenbud2013commutative}.
\end{remark}
In the following, we always identify $\calR$ with the coordinate ring of $\Pn = \bbP(\calS_1)$. 

\begin{definition} \label{def:apolar_scheme}
    Let $F \in \calS_d$. A $0$-dimensional scheme $Z \subset \Pn$ is said to be \textbf{apolar} to $F$ if $I(Z) \subseteq \Ann(F)$.
\end{definition}

A famous characterization of schemes apolar to a given form is provided by the well-known {\it Apolarity Lemma}, see e.g. \cite[Lemma 1.15]{IaKa:book} in the classical case of reduced schemes, \cite[Lemma 1]{BJMK18:Polynomials} for non-reduced scheme or \cite[Lemma 1.3]{ranestad2018varieties} into a more general framework.
\begin{lemma}[Apolarity Lemma]\label{lemma:apolarity}
    Let $F \in \calS_d$. The following are equivalent:
    \begin{itemize}
        \item $F \in I(Z)^\perp_d$;
        \item $I(Z) \subset \Ann(F)$.
    \end{itemize}
\end{lemma}

\begin{remark} \label{rmk:dpornotdp}
    We will construct local zero-dimensional projective schemes apolar to a given form $F \in \calS_d$ and supported at the point $[L] \in \Pn$, called \textit{natural apolar schemes} \cite{BJMK18:Polynomials}.
    We define them locally, by computing the annihilator of a dehomogenization $f$ with respect to $L$.
    However, the vanishing of $g \circ f$ is not preserved by homogenization (for example, $(y_1+y_2^3) \circ (6x_1-x_2^3) = 0$, but $(Y_0^2Y_1+Y_2^3) \circ (6X_0^2X_1-X_2^3) = 6$).
    Instead, the action that is preserved is the contraction on divided powers rings, see \eqref{action:contract}.
    Unlike \cite{BJMK18:Polynomials}, we prefer to keep in our constructions both actions for computational convenience.
    Indeed, when a general support $L$ is considered, we will reduce to the case of $L$ being a simple variable by a coordinate change, compute the (homogenization of the) annihilator ideal by contraction of $f$, and then consider the (dualization of the) inverse change of coordinates.
    The latter change of coordinates keeps the apolarity relation by derivation when naturally extended to any degree of the polynomial algebra, see \Cref{lem:apolar_changecoordinates}, while it should be defined degree-by-degree in order to maintain apolarity in the sense of contraction over divided powers. 
\end{remark}

Let $\calS_{\rm dp}$ be the polynomial ring $\calS$ equipped with a \textit{divided power structure}, i.e. endowed with the divided powers monomial basis $X^{[\alpha]} = \frac{1}{\alpha!}X^\alpha$.
We denote by $F_{\rm dp} \in \calS_{\rm dp}$ the polynomial $F \in \calS$ expressed in divided powers.

As discussed in \Cref{rmk:dpornotdp}, 
we also consider the action of $\calR$ on $\calS_{\rm dp}$ by contraction, i.e.
\begin{equation}\label{action:contract}
Y^\beta \contract X^{\alpha} = 
	\begin{cases}
		X^{\alpha-\beta} & \text{ if } \alpha \succeq \beta, \\
		0 & \text{ otherwise}.
	\end{cases}
\end{equation}

For a given $F \in \calS_{\rm dp}$, its annihilator with respect to this action will be denoted by
\[
    \Ann^{\contract}(F) = \left\{G \in \calR ~:~ G \contract F = 0\right\}.
\]
One can directly verify that $G \contract F_{\rm dp} = (G \circ F)_{\rm dp}$, so in particular $\Ann^{\contract}(F_{\rm dp}) = \Ann(F)$.

\subsection{Minimality} In this paper, we consider the $0$-dimensional schemes apolar to a given $F \in \calS_d$. Among them, we are particularly interested in those that are minimal by inclusion or length.

\begin{definition}\label{def:irredundant_scheme}
    Let $Z \subset \Pn$ be a $0$-dimensional scheme apolar to $F \in \calS_d$.
    We say that $Z$ is \textbf{irredundant} to $F$ if there is no strict subscheme $Z' \subsetneq Z$ among the schemes apolar to $F$.
\end{definition}

The minimal length of a $0$-dimensional scheme apolar to $F$ is called both \textbf{scheme length} of $F$ \cite{IaKa:book} or \textbf{cactus rank} of $F$ \cite{BR13, BB14:Secant}.

\begin{definition}\label{def:minimal_scheme}
    Let $Z \subset \Pn$ be a $0$-dimensional scheme apolar to $F \in \calS_d$.
    We say that $Z$ \textbf{evinces the cactus rank}, or \textbf{evinces the scheme length} of $F$, or simply is \textbf{minimal apolar} to $F$, if $Z$ is of minimal length among the $0$-dimensional schemes in $\Pn$ and apolar to $F$. 
\end{definition}

\subsection{Regularity} We study when the Hilbert function of minimal apolar schemes stabilizes.

\begin{definition}
    Given a homogeneous ideal $I \subset \calR$, the \textbf{Hilbert function} of the quotient $\calR/I$ is the function $\HF_{\calR/I} : \bbN \to \bbN$ such that $\HF_{\calR/I}(i) = \dim \calR_i / I_i$, where $I_i = I \cap \calR_i$. For a scheme $Z \subset \Pn$ we denote the Hilbert function of $Z$ as $\HF_Z = \HF_{\calR/I(Z)}$.
\end{definition}

We simply write $\HF_Z = (a_0,a_1,a_2,\dots)$ to denote $\HF_Z(i) = a_i$.


The Hilbert function of a $0$-dimensional scheme $Z$ is always strictly increasing until it reaches its length $\len(Z)$, and then it remains constant. 
\begin{definition}
    Given a $0$-dimensional scheme $Z \subset \Pn$, the \textbf{regularity} of $Z$ is
    \[\reg(Z) = \min_{i \in \N}\{\HF_Z(i) = \HF_Z(i+1) = \len(Z)\}.\]
    We say that $Z$ is \textbf{regular in degree $d$}, or \textbf{$d$-regular}, if $\reg(Z) \leq d$.
\end{definition}


\section{Schemes evinced by GADs} \label{sec:GAD}


We devote the present section to connecting two well-known concepts such as natural apolar schemes \cite{BJMK18:Polynomials} and generalized additive decomposition \cite{IaKa:book}.
Their link serves as the cornerstone of our paper, while their explicit construction may be beneficial even for expert readers.
A complete implementation in Macaulay2 \cite{M2} and Magma \cite{MR1484478} of these procedures may be found in \cite{Repo}.


\subsection{Natural apolar scheme to F supported at L}\label{sec:NAS}
There is a natural way to associate a local scheme apolar to a given $F \in \calS_d$  supported at a prescribed point $[L]\in \Pn$ \cite[Section~4]{BJMK18:Polynomials}.
Let $f_L \in \calS_{\rm dp} / (L-1) = \calSa$ be the dehomogenization of $F_{\rm dp}$ by $L$.
We consider the projection $ \calS_{\rm dp} \to \calSa$ and its dual projection $\calR \to \calRa$.
We denote the latter projection of an ideal $J \subset \calR$ by $\underline{J} \subset \calRa$.
We will always use lowercase letters for the elements and the variables after these projections, e.g., we identify $\calSa \simeq \Bbbk[x_1,\ldots,x_n]_{\rm dp}$ and $\calRa \simeq \Bbbk[y_1,\ldots,y_n]$.

\begin{definition}\label{def:naturalApolar} 
Let $F \in \calS_d$ and $L \in \calS_1$.
We define the \textbf{natural apolar scheme to $F$ supported at $L$} the scheme $Z_{F,L} \subset \Pn$ supported at $[L] \in \Pn$ and locally defined by $\underline{I}(Z_{F,L}) = \Ann^{\contract}(f_{L}) \subset \calRa$.
\end{definition}

Note that $\calRa$ can be regarded as the coordinate ring of the affine chart $U_0 = \{[L] ~:~ Y_0 \circ L \neq 0\} \subset \Pn$ and $Z_{F,L}$ is a local $0$-dimensional scheme supported at the origin of~$U_0$. 

Contraction behaves well with dehomogenization with respect to dual variables.
In particular, if $g \in \calRa$ is the dehomogenization of $G \in \calR$ with respect to $Y_0$, and $g \contract f_{X_0} = 0$, then $G \contract F_{\rm dp} = 0$ \cite[Corollary 3]{BJMK18:Polynomials}, and the last equality implies that $G \circ F = 0$ as observed in \Cref{sec:apolarity}.
Hence, the scheme $Z_{F,L}$ is apolar to $F$ according to \Cref{def:apolar_scheme}.
\begin{lemma}[{\cite[Corollary 4]{BJMK18:Polynomials}}]\label{lemma:apolar_local}
    The scheme $Z_{F,L}$ is apolar to $F$.
\end{lemma}


Here we detail how to concretely construct the ideal defining such a scheme.


Fix $F \in \calS_d$ and $L = \ell_0X_0 + \ldots + \ell_nX_n \in \calS_1$. 
Without loss of generalities 
we may assume $\ell_0 = 1$. Over $\calS$, we consider the change of variables given by 
\begin{equation}\label{eq:Xchange}
    \phi : \calS \rightarrow \calS, \qquad \begin{cases}
        X_0 \mapsto X_0 - \sum_{i=1}^n \ell_iX_i, \\ X_i \mapsto X_i, & \text{ for } i \in \{1,\ldots,n\}.
    \end{cases}
\end{equation} 
We have $\phi(L) = X_0$ and $\tilde F = \phi(F)$, 
therefore we represent $f_L$ as $\tilde f_{X_0} = \tilde F_{\rm dp}(1,x_1,\ldots,x_n) \in \calSa$.
Then $\Ann^{\contract}(f_{L})$ is the kernel of the infinite-dimensional \emph{Hankel operator} \cite{BRACHAT20101851,BT20:Waring}:
\[
    H(f_{L}) : \calRa \to \calSa, \quad g \mapsto g \contract f_{L}.
\]

However, since $y^\beta \contract f_{L} = 0$ for every $|\beta| > \deg( f_{L})$, the annihilator of $f_{L}$ is generated by the kernel of a truncated Hankel operator.
Let $e = \deg(f_{L})$ and consider the restriction
\[
    H^{e+1}(f_{L}) : \calRa_{\leq e+1} \to \calSa_{\leq e}.
\]
Then, $\Ann^{\contract}(f_{L}) = \ker H^{e+1}(f_{L}).$

Note that the coefficients of the Hankel matrix can be computed directly from $\tilde F$. Indeed, if we label rows and columns of $H^{e+1}(f_{L})$ accordingly to the divided powers monomial basis of $\calSa_{\leq e}$ and the standard monomial basis of $\calRa_{\leq e+1}$, respectively, we have 
\begin{equation}\label{eq:hankel}
    [H^{e+1}(f_{L})]_{\alpha,\beta} = {\rm eval}_{(0,\ldots,0)}\left(y^{\alpha+\beta} \contract f_{L}\right) = \begin{cases}
        Y^{(d-(|\alpha|+|\beta|),\alpha_1+\beta_1,\cdots,\alpha_n+\beta_n)} \circ \tilde F  & \text{ if }|\alpha|+|\beta| \leq d, \\
        0 & \text{ otherwise.}
    \end{cases}
\end{equation}

The next proposition shows that $\Ann^{\contract}(f_L)$ can almost always be generated in degree $d$, precisely when $f_L$ cannot be written as a univariate polynomial in the dp-algebra of $\calSa$.
The following proposition arose from a private communication with J. Jelisiejew.

\begin{proposition}[J. Jelisiejew] \label{prop:Annf}
    Let $f \in \calSa$ with $\deg(f) = e > 0$. The following are equivalent:
    \begin{enumerate}
        \item \label{Annf-i} $\Ann^{\contract}(f) \neq \left( \Ann^{\contract}(f)_{\leq e} \right)$.
        \item \label{Annf-ii} $\left( \Ann^{\contract}(f)_{\leq e} \right) \not\supset \calRa_{e+1}$.
        \item \label{Annf-iii} There are $l \in \calSa_1$ and $\lambda_0, \dots, \lambda_{e-1} \in \Bbbk$ such that
        \[ f = l^e + \sum_{i=0}^{e-1} \lambda_i l^i. \]
        \item \label{Annf-iv} $\dim \langle \calRa_1 \contract f \rangle = 1$.
    \end{enumerate}
\end{proposition}
\proof \underline{$\eqref{Annf-i} \Rightarrow \eqref{Annf-ii}$.} Clearly $\Ann^{\contract}(f) \supseteq \left( \Ann^{\contract}(f)_{\leq e} \right)$ and these ideals agree up to degree $e$.
If by contradiction $\left( \Ann^{\contract}(f)_{\leq e} \right) \supset \calRa_{e+1}$, then they agree also in degrees greater than $e$, hence they are equal.

\underline{$\eqref{Annf-ii} \Rightarrow \eqref{Annf-iii}$.} Let $f = f_e + f_{\leq e-1}$ with $f_e \in \calSa_e$ and $f_{\leq e-1} \in \calSa_{\leq e-1}$. We split \eqref{Annf-iii} in two parts.
\begin{claim}\label{claim1}
    There exists $l \in \calSa_1$ such that $f_e = l^e$.
\end{claim}
\begin{proof}[Proof of \Cref{claim1}]
Since $f_e$ is homogeneous we can see it as an element of $\calS_{\textrm{dp}}$, which isomorphic as $R$-module to $\calS$ \cite[Theorem 9.5]{Ger96}. Hence, by \cite[Proposition 1.6]{BBKT} we have 
\[ \exists l \in \calSa_1 \ : \ f_e = l^e \quad \iff \quad \Ann^{\contract}(f_e) \neq \left( \Ann^{\contract}(f_e)_{\leq e} \right). \]
Assume by contradiction that $\Ann^{\contract}(f_e) = \left( \Ann^{\contract}(f_e)_{\leq e} \right)$, then $\calRa_{e+1} = \calRa_1 \cdot \big( \Ann^{\contract}(f_e) \cap \calRa_e \big)$.
Since homogeneous forms in $\calRa_e$ annihilate $f$ if and only if they annihilate $f_e$, we have
\begin{equation} \label{eq:deg_e+1}
    \Ann^{\contract}(f_e) \cap \calRa_e = \Ann^{\contract}(f) \cap \calRa_e.
\end{equation}
We conclude that
\[ \calRa_{e+1} = \calRa_1 \cdot \big( \Ann^{\contract}(f) \cap \calRa_e \big) \subseteq \left( \Ann^{\contract}(f)_{\leq e} \right), \]
contradicting \eqref{Annf-ii}.
\end{proof}

\begin{claim}\label{claim2}
    There are $\lambda_0, \dots, \lambda_{e-1} \in \Bbbk$ such that $f_{\leq e-1} = \sum_{i=0}^{e-1} \lambda_i l^i$.
\end{claim}
\begin{proof}[Proof of \Cref{claim2}]
Complete $x_1^{(l)} = l$ to a basis $\{ x_1^{(l)}, \dots, x_n^{(l)} \}$ of $\calSa_1$ and let $\{ y_1^{(l)}, \dots, y_n^{(l)} \}$ be a dual basis of $\calRa_1$, i.e. $y_j^{(l)} \contract x_i^{(l)} = \delta_{ij}$.
We regard $f_{\leq e-1}$ as a polynomial in $\Bbbk[x_1^{(l)}, \dots, x_n^{(l)}]$ endowed with a graded anti-lexicographic order, i.e. $x_1^{(l)} < \dots < x_n^{(l)}$, and let $k_{\a} (x^{(l)})^{\a}$ be its leading term.
Moreover, we define
\[ \sigma = \frac{(y^{(l)})^\a}{k_{\a}} \in \calRa_{\leq e-1}. \]
Clearly $\sigma \contract f_{\leq e-1} = 1$ by definition.
Let us assume by contradiction that $f_{\leq e-1}$ cannot be written as a univariate polynomial in $x_1^{(l)}$, then there is $1 < m \leq n$ such that $y_m^{(l)} \ | \ (y^{(l)})^\a$, so $\sigma \contract f_{e} = \sigma \contract (x_1^{(l)})^e = 0$.
Therefore, $\sigma \contract f = \sigma \contract f_{\leq e-1} = 1$, so for every $1 \leq i \leq n$ we have
\[ y_i^{(l)}\sigma \in \Ann^{\contract}(f)_{\leq e}. \]
Let $\tau = (y_1^{(l)})^e \in \calRa_e$.
Since $\tau \contract f = \tau \contract f_e = 1$, we also have $(\tau-\sigma) \contract f = 0$, hence
\[ y_i^{(l)}\tau = y_i^{(l)}\sigma + y_i^{(l)}(\tau - \sigma) \in \left( \Ann^{\contract}(f)_{\leq e} \right). \]
By \cref{eq:deg_e+1} we have $\codim_{\calRa_e}(\Ann^{\contract}(f) \cap \calRa_e) = 1$, then
\[ \calRa_e = \big( \Ann^{\contract}(f) \cap \calRa_e \big) + \Bbbk \tau, \]
from which we obtain
\[ \calRa_{e+1} = \calRa_1 \cdot \calRa_e = \calRa_1 \cdot \big( \Ann^{\contract}(f) \cap \calRa_e \big) + \calRa_1 \cdot \tau \subseteq \left( \Ann^{\contract}(f)_{\leq e} \right), \]
contradicting \eqref{Annf-ii}.
\end{proof}

\underline{$\eqref{Annf-iii} \Rightarrow \eqref{Annf-iv}$.} With the choice of coordinates as above, we have $\langle y_2^{(l)}, \dots, y_n^{(l)} \rangle \contract f = 0$, while $y_1^{(l)} \contract f = l^{e-1} + \sum_{i = 0}^{e-2} \lambda_{i+1}l^i \neq 0$, since $e \geq 1$.

\underline{$\eqref{Annf-iv} \Rightarrow \eqref{Annf-i}$.} Up to permuting the variables of $\calRa$, we can assume $y_1 \contract f \neq 0$.
Since we have $\dim \langle \calRa_1 \contract f \rangle = 1$, there are $\mu_2, \dots, \mu_n \in \Bbbk$ such that $y_i \contract f = \mu_i y_1 \contract f$.
Therefore $\big(y_i-\mu_i y_1 \big)_{2 \leq i \leq n} \subset \Ann^{\contract}(f)$, hence $f$ can be regarded as a degree-$e$ univariate polynomial in $l = x_1 + \mu_2 x_2 + \dots + \mu_n x_n$.
We conclude that
$\Ann^{\contract}(f) = \big(y_i-\mu_i y_1, y_1^{e+1} \big)_{2 \leq i \leq n}$, but 
$y_1^{e+1} \not\in (y_i-\mu_i y_1)_{2 \leq i \leq n} = \left( \Ann^{\contract}(f)_{\leq d} \right)$.
\endproof

\begin{remark} \label{rmk:degf}
By \Cref{prop:Annf}, whenever $f_L$ cannot be regarded as a univariate polynomial, we may compute $\Ann^{\contract}(f_L)$ by restricting its Hankel matrix to $H^{e}(f_L) : \calRa_{\leq e} \to \calSa_{\leq e}$, which makes the kernel computation more efficient, see e.g. \Cref{ex:natural,ex:bad}.
By part \eqref{Annf-iv} of the same proposition, we can test this condition by checking whether the rank of the submatrix of $H^{e}(f_L)$ made by the lines indexed by $x_1, \dots, x_n$ has rank $1$. \\
Although uncommon in general, we notice that in some families of GADs (such as tangential decompositions, see \Cref{sec:tangential}, or when $n=1$) univariate $f_L$'s naturally arise. In such cases, we actually need to consider the kernel of the rectangular matrix $H^{e+1}(f_L)$ to fully recover $\Ann^{\contract}(f_L)$, see e.g. \Cref{ex:gadcomputation}.
\end{remark}

The homogenization $\tilde I = I(Z_{\tilde F,X_0}) = [\Ann^{\contract}(f_{L})]^{\rm hom} \subset \calR$ with respect to $Y_0$ defines a $0$-dimensional scheme apolar to $\tilde F$ and supported at $[X_0] \in \Pn$ as in \Cref{def:naturalApolar}.

\begin{remark} \label{rmk:homog}
    Note that the ideal homogenization is the only step in which non-linear algebra (e.g. Gröbner bases) may be required.
    In fact, the ideal defining the projective closure of an affine variety $Y$ needs not to be generated by the homogenization of the generators of $I(Y)$ \cite[Exercise 2.9]{Hartshorne}.
    However, the homogenization of a Gröbner basis of $I(Y)$ with respect to any graded monomial ordering is a Gröbner basis for the projective closure with respect to the inherited monomial ordering (where the homogenizing variable is set smaller than the others), since in this setting the leading monomial does not change under homogenization \cite[Proposition 4.3.21]{KreuzerRobbiano}.
    Thus, in the following, we will always compute the homogenization of ideals by homogenizing one of its graded Gröbner bases.
\end{remark}

Finally, to obtain the ideal defining $Z_{F,L}$ as in \Cref{def:naturalApolar}, we need to support $\tilde I$ on $[L] \in \Pn$, hence we perform the change of coordinate in $\calR$ given by the dualization of the inverse of \cref{eq:Xchange}:
\begin{equation}\label{eq:Ychange}
    \psi = \phi^T : \calR \rightarrow \calR, \qquad
\begin{cases}
    Y_0 \mapsto Y_0, \\
    Y_i \mapsto -\ell_iY_0 + Y_i, &\text{ for } i \in \{1,\ldots,n\}.
\end{cases}
\end{equation}
The ideal $I = \psi(\tilde I) \subset \calR$ defines a $0$-dimensional scheme which is supported at $[L]$ and apolar to $F$.
Indeed, the following lemma shows that the action by derivation is preserved under the changes of coordinates given by \cref{eq:Xchange,eq:Ychange}.
\begin{lemma}\label{lem:apolar_changecoordinates}
Let $\phi$ and $\psi$ be changes of coordinates of \cref{eq:Xchange,eq:Ychange}. Then we have
\[
    \psi(Y^\beta) \circ \phi^{-1}(X^\alpha) = Y^\beta \circ X^\alpha.
\]
\end{lemma}
\begin{proof}
    We write
    \[ \psi(Y^{\beta}) \circ \phi^{-1}(X^\alpha) = \psi(Y_0^{\beta_0}) \circ \left( \psi(Y_1^{\beta_1}) \circ \left( \dots \psi(Y_n^{\beta_n}) \circ \phi^{-1}(X^\alpha)\right)\right). \]
    By the chain rule of derivation, if $L \circ M = 0$ then $L^b \circ M^a = 0$ for any $a,b \in \N$. 
    In particular, for every $j \in \{1, \dots, n\}$ we have
    \begin{align*}    
    \psi(Y_j^{\beta_j}) \circ \phi^{-1}(X^\alpha) &= (-\ell_jY_0+Y_j)^{\beta_j} \circ \left[\left(X_0+\ell_1X_1+\ldots+\ell_nX_n\right)^{\alpha_0}\prod_{i=1}^nX_i^{\alpha_i}\right] \\
    & = \left[(X_0+\ell_1X_1+\ldots+\ell_nX_n)^{\alpha_0}\prod_{\substack{1 \leq i \leq n \\ i\neq j}}X_i^{\alpha_i}\right]\cdot (Y_j^{\beta_j}\circ X_j^{\alpha_j}).
    \end{align*}
    Therefore, by repeatedly applying the above equation for every $j$ we obtain
    \[
        \psi(Y_1^{\beta_1}) \circ \left( \dots \psi(Y_n^{\beta_n}) \circ \phi^{-1}(X^\alpha)\right) = (X_0+\ell_1X_1+\ldots+\ell_nX_n)^{\alpha_0} \cdot (Y_1^{\beta_1}\circ X_1^{\alpha_1}) \cdots (Y_n^{\beta_n}\circ X_n^{\alpha_n}).
    \]
    The result follows by acting with $\psi(Y_0^{\beta_0}) = Y_0^{\beta_0}$ on the above quantity.
\end{proof}
Note that our choice of the change of coordinates in \cref{eq:Xchange} was arbitrary.
It would have been enough to consider any set of linear forms $\{L_1,\ldots,L_n\}$ completing a basis of $\calS_1$ together with $L$. 
Then, $\phi$ is the change of coordinates inverse to the one sending $X_0 \mapsto L_0$ and $X_i \mapsto L_i$, for any $i \in \{1,\ldots,n\}.$

\begin{alg}[Natural Apolar Scheme] \label{alg:NaturalApolarScheme}
Summary of construction of natural apolar schemes. 

\smallskip
\noindent\textbf{Input:} A homogeneous polynomial $F \in \calS_d$ and a linear form $L \in \calS_1$. \\
    \textbf{Output:} The ideal $I(Z_{F,L}) \se \calR$.
\begin{enumerate}
    \item Define $\tilde F$ as the base-change of $F$ as in \cref{eq:Xchange}.
    \item Compute $f_{L}$ as $\tilde F_{\rm dp}(1,x_1,\dots,x_n)$ and set $e = \deg ( f_{L} )$.
    \item \label{HankelStep} Compute the ideal $\underline{I} = \ker H^{e+1}( f_{L} )$ as in \cref{eq:hankel}.
    \item \label{GroebnerStep} Compute the homogenization $I \subset \calR$ of $\underline{I} \subset \calRa$ as in \Cref{rmk:homog}.
    \item Return the base-change of the ideal $I$ as in \cref{eq:Ychange}.
\end{enumerate}
\end{alg}

\medskip

\subsection{Generalized Additive Decompositions (GADs)}

We recall the definition of the so-called generalized additive decompositions as introduced in \cite{IaKa:book}, and we associate to them $0$-dimensional schemes by employing the notion of natural apolar scheme introduced in \Cref{sec:NAS}.

\begin{definition}\label{def:GAD}
    Let $F \in \calS_d$ and let $L_1,\ldots,L_s \in \calS_1$ be pairwise non-proportional linear forms. A \textbf{generalized additive decomposition} (GAD) of $F$ \textbf{supported at $\{L_1,\ldots,L_s\}$} 
    is an expression
	\begin{equation}\label{eq:GAD}
		F = \sum_{i=1}^s L_i^{d-k_i}G_i,
	\end{equation}
    where for every $i \in \{1,\ldots,s\}$ we have $0 \leq k_i \leq d$ and $G_i \in \calS_{k_i}$ is not divisible by $L_i$.
    If $s=1$, we call this GAD \emph{local}.
\end{definition}

Following \cite{BBM14:comparison, BJMK18:Polynomials}, we associate a $0$-dimensional scheme to any GAD as \cref{eq:GAD}.
\begin{definition}\label{def:GADscheme}
    The \textbf{scheme evinced by a GAD} as in \cref{eq:GAD} is the union of the natural apolar schemes to each summand with respect to the corresponding $L_i$, i.e., 
	\[
		Z = \bigcup_{i=1}^sZ_{L_i^{d-k_i}G_i, L_i}. 
    \]
    The \textbf{size} of a GAD as in \cref{eq:GAD} is the length of the evinced scheme $Z$.
\end{definition}

Note that the same scheme may be evinced by different GADs.
Indeed, $ L^{d-k}G$ and $
L^{d-k}G'$ evince the same scheme whenever $\Ann^{\contract}(g_{L}) = \Ann^{\contract}(g'_{L})$.
However, schemes evinced by GADs of a given $F$ are always apolar to it. 
\begin{lemma}
    Let $Z$ be the scheme evinced by a GAD of $F$. Then $Z$ is apolar to $F$.
\end{lemma}
\begin{proof}
    To ease notation, denote $F_i = L_i^{d-k_i}G_i$ in \cref{eq:GAD}. Let $I(Z)_d = I(Z) \cap \calR_d$ and let $I(Z)_d^\perp$ be the orthogonal space via the non-degenerate pairing $\calR_d \times \calS_d \to \Bbbk$ induced by derivation.
    Then,
    \[
	I(Z)^\perp_d = \left(I\left(Z_{F_1,L_1}\right)_d \cap \ldots \cap I\left(Z_{F_s,L_s}\right)_d\right)^\perp = I\left(Z_{F_1,L_1}\right)^\perp_d + \ldots + I\left(Z_{F_s,L_s}\right)_d^\perp,
    \]
    see e.g. \cite[Proposition 2.6]{Ger96}.
    For every $i \in \{1,\ldots,s\}$ we have $F_i \in I\left(Z_{F_i,L_i}\right)_d^\perp$ 
    by \Cref{lemma:apolar_local}. Hence $F \in I(Z)_d^\perp$, which by \Cref{lemma:apolarity} implies that $I(Z) \se \Ann(F)$.
\end{proof}

The ideal defining schemes evinced by GADs can be easily computed by intersecting the ideals defining natural apolar schemes to local pieces of the additive decomposition, computed as in \Cref{alg:NaturalApolarScheme}. 

\subsection{Examples} \label{app:Examples} Here we illustrate the above construction with some examples.

\begin{example}\label{ex:natural} 
    Let $F = (X_0+3X_1-2X_2)(X_1+X_2)X_2 \in \calS_3$ and $L = X_0+3X_1-2X_2 \in \calS_1$.
    Following \Cref{alg:NaturalApolarScheme} we obtain $\tilde F = X_0X_1X_2 + X_0X_2^2 \in \calS$ by 
\[
	X_0 \leftarrow X_0-3X_1+2X_2. 
\]
In divided powers it becomes $\tilde F_{\rm dp} = X_0X_1X_2 + 2X_0X_2^{[2]}$, whose de-homogenization by $X_0 = 1$ is equal to $f_{L} = x_1x_2+2x_2^{[2]} \in \calSa$.
Since $f_{L}$ cannot be written with only one variable, by \Cref{rmk:degf} we consider the truncation of the Hankel matrix in degree $2 = \deg ( f_{L} )$, i.e.,
\[
    H^{2}(f_{L}) = \begin{blockarray}{c c c c c c c c c c c}
	& 1 & y_1 & y_2 & y_1^2 & y_1y_2 & y_2^2  \\
    \begin{block}{c (c c c c c c c c c c)}
     1 & 0 & 0 & 0 & 0 & 1 & 2 \\
     x_1 & 0 & 0 & 1 & 0 & 0 & 0  \\
     x_2 & 0 & 1 & 2 & 0 & 0 & 0  \\
     x_1^2 & 0 & 0 & 0 & 0 & 0 & 0  \\
     x_1x_2 & 1 & 0 & 0 & 0 & 0 & 0  \\
     x_2^2 & 2 & 0 & 0 & 0 & 0 & 0  \\
    \end{block}
    \end{blockarray} \ ,
\]
whose kernel defines the ideal $\Ann^{\contract}(f_{L}) = \big( y_2(2y_1-y_2), y_1^2 \big) \subset \calRa$.
Its homogenization in $\calR$ is the ideal $\big( Y_2(2Y_1-Y_2), Y_1^2 \big)$, which we need to base-change as in \cref{eq:Ychange}, i.e.
\[
	Y_1 \leftarrow -3Y_0 + Y_1, \quad Y_2 \leftarrow 2Y_0 + Y_2.
\] 
This way we obtain the ideal $I = \left( (2Y_0+Y_2)(8Y_0-2Y_1+Y_2), (3Y_0-Y_1)^2 \right) \subset \calR$.
Its radical ideal is $( 2Y_1+3Y_2, 2Y_0+Y_2 )$, i.e. it defines a $0$-dimensional scheme supported at the point $[L] = [X_1+3X_2-2X_3] \in \Pn$.
One can directly verify that this scheme has length $4$, either by computing the dimension of the space of partials of $f_L$ or by looking at the Hilbert polynomial.
Furthermore, it is apolar to $F$, as we have
\begin{align*}
			(2Y_0+Y_2)(8Y_0-2Y_1+Y_2) \circ F & = (-16Y_0^{2}+4Y_0Y_1-10Y_0Y_2+2Y_1Y_2-Y_2^{2}) \circ F = 0, \\
			(3Y_0-Y_1)^2 \circ F & = (9Y_0^2 - 6Y_0Y_1 + Y_1^2) \circ F = 0.
\end{align*}
Hence, $I$ is the ideal defining $Z_{F,L}$. \exampleend
\end{example}


\begin{example}\label{ex:bad} 
    Let $F = (X_0+3X_1-2X_2)(X_1+X_2)X_2 \in \calS_3$ the same polynomial of \Cref{ex:natural} and consider $L = X_0 \in \calS_1$.
    As the support is $X_0$, we do not need to change coordinates, so we directly de-homogenize $F_{\rm dp}$ with respect to $L$, obtaining $f_{L} = x_1x_2 + 2x_2^{[2]} + 6x_1^{[2]}x_2 + 2x_1x_2^{[2]} - 12x_2^{[3]} \in \calSa$.
    Since $f_{L}$ cannot be regarded as univariate polynomial, we consider the truncation of the Hankel matrix in degree $3 = \deg(f_L)$, namely
\[
    H^{3}(f_{L}) = \begin{blockarray}{c c c c c c c c c c c c c c c}
	& 1 & y_1 & y_2 & y_1^2 & y_1y_2 & y_2^2 & y_1^3 & y_1^2y_2 & y_1y_2^2 & y_2^3 \\
    \begin{block}{c (c c c c c c c c c c c c c c)}
      1 &       0  & 0  & 0  & 0  & 1  & 2  & 0  & 6  & 2 &-12 \\
      x_1 &     0  & 0  & 1  & 0  & 6  & 2  & 0  & 0  & 0  & 0 \\
      x_2 &     0  & 1  & 2  & 6  & 2 &-12  & 0  & 0  & 0  & 0 \\
      x_1^2 &   0  & 0  & 6  & 0  & 0  & 0  & 0  & 0  & 0  & 0 \\
      x_1x_2 &  1  & 6  & 2  & 0  & 0  & 0  & 0  & 0  & 0  & 0 \\
      x_2^2 &   2  & 2 &-12  & 0  & 0  & 0  & 0  & 0  & 0  & 0 \\
      x_1^3 &   0  & 0  & 0  & 0  & 0  & 0  & 0  & 0  & 0  & 0 \\
      x_1^2x_2 &6  & 0  & 0  & 0  & 0  & 0  & 0  & 0  & 0  & 0 \\
      x_1x_2^2 &2  & 0  & 0  & 0  & 0  & 0  & 0  & 0  & 0  & 0 \\
      x_2^3 &   -12  & 0  & 0  & 0  & 0  & 0  & 0  & 0  & 0  & 0 \\
    \end{block}
    \end{blockarray} \ .
\]
Its kernel is given by the ideal
\[
    \Ann^{\contract}(f_{L}) = ( 5y_2^3 + 76y_1^2 - 12y_1y_2 + 36y_2^2,
    2 y_1^2y_2 + y_2^3, 
    y_1^3,
    6y_1y_2^2 + y_2^3 ) \subset \calRa.
\]
To homogenize it, we compute a Gröbner basis with respect to the graded lexicographic ordering:
\[
    \Ann^{\contract}(f_{L}) = ( y_1^3,
    5y_1^2y_2 - 38y_1^2 + 6y_1y_2 - 18y_2^2,
    15y_1y_2^2 - 38y_1^2 + 6y_1y_2 - 18y_2^2,
    5y_2^3 + 76y_1^2 - 12y_1y_2 + 36y_2^2 ).
\]
Hence the natural apolar scheme $Z_{F,L}$ is defined by the ideal
\begin{equation*}
    \left( 
    Y_1^3, 2 Y_1^2 Y_2 + Y_2^3, 6 Y_1^2 Y_2 + Y_2^3,
    76 Y_0 Y_1^2 - 12 Y_0 Y_1 Y_2 + 36 Y_0 Y_2^2 + 5 Y_2^3  
    \right)
    \subset \calR.
\end{equation*}
One can easily verify that this ideal indeed defines a $0$-dimensional scheme apolar to $F$ and supported at $[X_0]$, whose length is $6$.
\exampleend
\end{example}

\begin{example} \label{ex:gadcomputation}
    Let $F = (X_0+3X_1-2X_2)(X_1+X_2)X_2 \in \calS_3$ be the polynomial of \Cref{ex:natural}.
    From the equality $(X_1 + X_2)X_2 = (\frac{X_1}{2}+X_2)^2-(\frac{X_1}{2})^2$ we immediately get another non-local GAD of $F$, namely
    \begin{equation}\label{eq:gad2}
	F = \left(\frac{X_1}{2}+X_2\right)^2 (X_0+3X_1-2X_2)-\left(\frac{X_1}{2}\right)^2 (X_0+3X_1-2X_2).
    \end{equation}
    We compute the scheme $Z$ evinced by the above GAD, supported at $[X_1+2X_2]$ and $[X_1]$.
 
    We begin with the first addendum $F_1 = \frac{1}{4}(X_1+2X_2)^2(X_0+3X_1-2X_2)$ and $L_1 = X_1+2X_2$.
    We can neglect the constant factor $\frac{1}{4}$, and since $L_1$ has no $X_0$ terms, we simply switch the roles of $X_0$ and $X_1$.
    In order to de-homogenize with respect to $L_1$, we perform the substitution
	\[
	       \quad X_1 \leftarrow X_1-2X_2,
	\]
    and we get $(f_1)_{L_1} = x_0+3-8x_2$.
    This polynomial can be written with one linear variable, as it is linear, so by \Cref{rmk:degf} we need to consider the truncation of the Hankel matrix in degree $2 = \deg\big(( f_1)_{L_1}\big) + 1$, i.e.
	\[
    H^2 \big( (f_1)_{L_1} \big) = \begin{blockarray}{c c c c c c c }
	& 1 & y_0 & y_2 & y_0^2 & y_0y_2 & y_2^2  \\
    \begin{block}{c (c c c c c c)}
     1 & 3 & 1 & -8 & 0 & 0 & 0  \\
     x_0 & 1 & 0 & 0 & 0 & 0 & 0  \\
     x_2 & -8 & 0 & 0 & 0 & 0 & 0  \\
    \end{block}
    \end{blockarray}\,,
    \]
    whose kernel defines the ideal $( 8y_0+y_2,y_2^2 ) \subset \calRa$.
    After the homogenization and the base-change
    \[
	Y_2 \leftarrow -2Y_1+Y_2,
    \]	
    we obtain the ideal $\left( 8Y_0-2Y_1+Y_2, (2Y_1-Y_2)^2 \right) \subset \calR$ defining the scheme $Z_1 = Z_{F_1,X_1+2X_2}$, which is $0$-dimensional, of length $2$ and supported at the point $[X_1+2X_2] \in \Pn$.
    
    We proceed with the second addendum $F_2 = \frac{1}{4}X_1^2(X_0+3X_1-2X_2)$ and $L_2 = X_1$.
    As above, $X_1$ plays the role of $X_0$.
    Since $(f_2)_{L_2} = x_0+3-2x_2$, we again consider the truncation of the Hankel matrix in degree $2$:
    \[
    H^2 \big( (f_2)_{L_2} \big) = \begin{blockarray}{c c c c c c c }
	& 1 & y_0 & y_2 & y_0^2 & y_0y_2 & y_2^2  \\
    \begin{block}{c (c c c c c c)}
     1 & 3 & 1 & -2 & 0 & 0 & 0  \\
     x_0 & 1 & 0 & 0 & 0 & 0 & 0  \\
     x_2 & -2 & 0 & 0 & 0 & 0 & 0  \\
    \end{block}
    \end{blockarray}\,,
    \]
    whose kernel defines the ideal $( 2y_0+y_2,y_2^2 ) \subset \calRa$.
    Hence, the scheme $Z_2 = Z_{F_2,X_1}$ is defined by the ideal $( 2Y_0+Y_2,Y_2^2 ) \subset \calR$, it is $0$-dimensional of length $2$ and is supported at the point $[X_1] \in \Pn$.
    In conclusion, the GAD of \cref{eq:gad2} evinces the scheme $Z = Z_1 \cup Z_2$ defined by
    \[
	    I(Z) = \big( 8Y_0-2Y_1+Y_2, (2Y_1-Y_2)^2 \big) \cap ( 2Y_0+Y_2,Y_2^2 ) = ( 4Y_0Y_1-10Y_0Y_2+2Y_1Y_2-Y_2^2, Y_0^2 ).
    \]
    One can directly check that $Z$ has length $4$, it is supported at the points $[X_1]$ and $[X_1+2X_2]$, and it is apolar to $F$.
    \exampleend
\end{example}


\section{GADs and Irredundant schemes} \label{sec:irred}


In this section, we investigate irredundant schemes evinced by GADs by employing ideas on natural apolar schemes from \cite{BJMK18:Polynomials}.

\begin{remark}\label{rmk:localGAD1}
    Let $[L] \in \Pn$ be a simple point defined by the ideal $\wp_L \subset \calR$. Recall that the \textit{$j$-fat point supported at $[L]$} is the $0$-dimensional scheme defined by the ideal $\wp_L^j$
    For any $k \leq d$, the natural apolar scheme of $F = L^{d-k}G \in \calS_d$ supported at $[L]$ is contained in the $(k+1)$-fat point supported at $[L]$, since the localization $f_{L}$ has degree at most $k$.
    Thus, $Z = Z_{F,L}$ is $k$-regular, as a $(k+1)$-fat point is always $k$-regular and the containment preserves regularity \cite{BCGI07:Osculating, Bernardi2009982}.
    Finally, if $F$ is \textit{concise} in $n+1$ variables, i.e., $\HF_Z(1) = n+1$, then $Z$ is regular in degree $\len(Z)-n$ since its Hilbert function starts with $\HF_Z = (1,n+1,\dots)$ and is strictly increasing until it stabilizes to $\len(Z)$.
\end{remark}

\begin{remark}\label{rmk:BJMRlemmas}
    By \cite[Lemma 3]{BJMK18:Polynomials}, given a local scheme $Z \subset \Pn$ apolar to $F \in \calS_d$ and supported at $[L]$, there exists $G \in \calS_D$ ($D \geq d$) such that $Z_{G,L} \se Z$ and $F = H \circ G$ for some $H \in \calR_{D-d}$.
    Furthermore, in \cite[Proposition 1]{BJMK18:Polynomials} it is shown that, under minimality assumption, the localizations of $F_{\rm dp}$ and $G_{\rm dp}$ with respect to $L$ are equal up to degree $d$.
    In that result, the minimality requirement is in terms of minimal length among the schemes supported at $[L]$ and apolar to $F$.
    However, we observe that in that proof irredundancy is actually enough.
    For the sake of completeness, we report here the proof of the following statement, which may be seen as a non-local version of \cite[Proposition 1]{BJMK18:Polynomials}. 
\end{remark}

\begin{proposition} \label{prop:ContainsGAD}
Let $Z$ be a $0$-dimensional scheme apolar and irredundant to $F \in \calS_d$.
Then $Z$ is evinced by a GAD of $F$ if and only if there are $L_1,\dots,L_s \in \calS_1$ such that $I(Z) \supseteq \bigcap_{i=1}^s \wp_{L_i}^{d+1}$.
\end{proposition}
\proof Let $Z = Z_1 \cup \dots \cup Z_s$ be the irreducible decomposition of $Z$.

If $Z$ is evinced by a GAD as in \cref{eq:GAD}, then each $Z_i = Z_{L_i^{d-k_i}G_i}$ is contained in a $(k_i+1)$-fat point by \Cref{rmk:localGAD1}, hence $I(Z_i) \supseteq \wp_{L_i}^{k_i+1} \supseteq \wp_{L_i}^{d+1}$. 
Note that this implication does not need irredundancy.

Conversely, since $I(Z) \se \Ann(F)$, then we have
\[
    F \in I(Z)^\perp_d = I(Z_1)_d^\perp + \dots + I(Z_s)^\perp_d.
\]
Therefore, we have an additive decomposition $F = \sum_{i=1}^s F_i$ with $F_i \in I(Z_i)^\perp_d$.
By \Cref{rmk:BJMRlemmas} there are $G_i \in \calS_{D_i}$ and $H_i \in \calR_{D_i-d}$ such that $Z_{G_i,L_i} \se Z_i$ and $H_i \circ G_i = F_i$. 
By \cite[Lemma 3]{BJMK18:Polynomials} we know that $h_i \contract (g_i)_{L_i}$ and $(f_i)_{L_i}$ are equal up to degree $d$, but since $I(Z_i) \supseteq \wp_{L_i}^{d+1}$, the degree of the local generator $h_i \contract (g_i)_{L_i}$ is bounded by $d$, so it equals $(f_i)_{L_i}$.
Hence, we have
\[\Ann^{\contract}\big((f_i)_{L_i}\big) = \Ann^{\contract}\big(h_i \contract (g_i)_{L_i}\big) \supseteq \Ann^{\contract}\big((g_i)_{L_i}\big),\]
therefore the natural apolar scheme $Z_{F_i,L_i}$ is contained in $Z_{G_i,L_i}$ and is apolar to $F_i$.
However, the scheme $Z_i$ needs to be irredundant to $F_i$, hence we conclude that $Z_{F_i,L_i} = Z_{G_i,L_i} = Z_i$.
Therefore $Z$ is the scheme evinced by the additive decomposition $\sum_{i=1}^s F_i$ supported at $L_1,\dots,L_s$.
\endproof

In the following example, we observe that even if $Z$ is evinced by a GAD of $F \in \calS_d$ and its components are contained in $(d+1)$-fat points, $Z$ may still be redundant to $F$.

\begin{example}
\label{ex:irredbutnotmin}
    Consider the GAD $F = X_0G_1 + X_1^2G_2 \in \calS_3$, where
    \[ G_1 = 4 X_0^2 + 2 X_0 X_1 - 4 X_1^2 , \quad G_2 = - 3 X_0 - 5 X_1. \]
The scheme $Z$ evinced by such GAD is given by the ideal 
\[
    I(Z) = ( Y_0^2 Y_1^3 ) = \wp_{X_0}^3 \cap \wp_{X_1}^2 \subset \calR.
\]
Its Hilbert function is $\HF_Z = (1,2,3,4,5,5,\dots)$, hence it is not $3$-regular.
We move the addendum in $G_1$ containing $X_1^2$ to $G_2$, obtaining a different GAD supported at the same points: $F = X_0^2 \tilde G_1 + X_1^2 \tilde G_2$, where
\[ \tilde G_1 =  4 X_0+2 X_1, \quad \tilde G_2 = -7 X_0 - 5 X_1. \]
The scheme $\tilde Z$ evinced by the last GAD is by construction apolar to $F$, and it is defined by 
\[
    I(\tilde Z) = ( Y_0^2 Y_1^2 ) = \wp_{X_0}^2 \cap \wp_{X_1}^2 \subset \calR.
\]
The Hilbert function of $\tilde Z \neq Z$ is $\HF_{\tilde Z} = (1,2,3,4,4,\dots)$, and clearly $I(Z) \subsetneq I(\tilde Z) \se \Ann(F)$.
Hence, $Z$ is redundant.

It can be directly verified that $\tilde Z$ is irredundant to $F$ (e.g. as in \Cref{ex:perazzo}), but not minimal, since
\[ I_W = ( 79 Y_0^2 - 166 Y_0Y_1 + 88 Y_1^2 ) \subset \calR \]
is evinced by the (unique) Waring decomposition of $F$, defining a scheme of length $2$ apolar to $F$.
\exampleend
\end{example}

\begin{corollary}\label{GAD:inside:fat:points}
    Let $L_1,\ldots,L_s \in \calS_1$ and $Z = Z_1 \cup \ldots \cup Z_s$ be a $0$-dimensional scheme apolar to $F \in \calS_d$ such that for every $i \in \{1, \dots, s\}$ we have $I(Z_i) \supset \wp_{L_i}^{\tilde k_i+1}$ with $\tilde k_i \leq d$.
    Then $Z$ contains a scheme evinced by a GAD of $F$ as in \cref{eq:GAD}, with $k_i \leq \tilde k_i$.
\end{corollary}
\proof
    Let $Y = Y_1 \cup \ldots \cup Y_s \subseteq Z$ be non-redundant and apolar to $F$, with $Y_i \subseteq Z_i$. Then, it is enough to apply the proof of \Cref{prop:ContainsGAD} to $Y$, since 
    \[ I(Y_i)^\perp_d \subseteq I(Z_i)^\perp_d \subseteq (\wp_i^{\tilde k_i+1})_d^\perp = \langle L_i^{d-{\tilde k_i}}Q ~:~ Q \in \calS_{\tilde k_i}\rangle, \]
    where the last equality is a classical result, see e.g. \cite[Theorem 3.2]{Ger96}.
    We conclude that $I(Y_i)$ is evinced by $F_i = L_i^{d-{\tilde k_i}} Q_i$, which becomes a valid (local) GAD after collecting all the factors $L_i$ in $Q_i$.
    Thus, $Y$ is evinced by the GAD $F = \sum_{i=1}^s F_i$ supported at $L_1, \dots, L_s$.
\endproof

In the following example, we show that the degree $D$ of the polynomial $G$ from \Cref{rmk:BJMRlemmas} may well exceed $d = \deg(F)$. We thank J. Buczyński for pointing it out.
\begin{example}\label{ex:GAD_nonContained}
     Consider the following polynomial:
    \begin{align*}
        F = 24\,X_{0}^{3}& + 70\,X_{0}^{2}X_{1} + 75\,X_{0}^{2}X_{2} + 70\,X_{0}^{2}X_{3} + 180\,X_{0}^{2}X_{4} +10\,X_{0}^{2}X_{5} + 10\,X_{0}X_{1}^{2} \\
        & + 70\,X_{0}X_{2}^{2}+ 360\,X_{0}X_{2}X_{3} + 120\,X_{0}X_{2}X_{4} + 60\,X_{0}X_{3}^{2} + 60\,X_{2}^{3} + 60\,X_{2}^{2}X_{3} \in \calS_3,
    \end{align*}
        and let $Z$ be the scheme defined by the ideal
        \begin{align*}
            I(Z) = ( -Y_0Y_3 + Y_2^2, \, &-Y_1Y_4 + Y_2Y_3, \, -Y_1Y_5 + Y_1^2, \, -6Y_1Y_5 + Y_2Y_4, \, -6Y_1Y_5 + Y_3^2, \\
            &Y_1Y_2, \, Y_1Y_3, \, Y_1Y_4, \, Y_1Y_5, \, Y_2Y_5, \, Y_3Y_4, \, Y_3Y_5, \, Y_4^2, \, Y_4Y_5, \, Y_5^2 ) \subset \calR.
        \end{align*}
        One can computationally check that $Z$ is a local $0$-dimensional
        scheme apolar to $F$, of minimal length $6$ and supported at $[X_0] \in \Pn$.
        One can also verify that it is the unique scheme of minimal length apolar to such $F$, by explicitly computing minimal apolar schemes \cite{BT20:Waring}, or by observing that $I(Z) = \Ann(F) \cap \calR_{\leq 2}$ and the Hilbert function of $\calR/\Ann(F)$ is $(1,6,6,1)$. In particular, $Z$ is non-redundant.
        Since $I(Z) \not\supseteq \wp_{X_0}^4$, by \Cref{prop:ContainsGAD} there is no GAD of $F$ that evinces this apolar scheme.
        However, as recalled in \Cref{rmk:BJMRlemmas}, since $I(Z) \supseteq \wp_{X_0}^5$ then $Z$ is evinced by a GAD of a degree-$4$ polynomial $G$ having $F$ among its partials. Indeed, let us consider the polynomial
        \begin{align*}
            G = \ &6 X_{0}^{4} + \frac{70}{3}X_{0}^{3}X_{1} + 25 X_{0}^{3}X_{2} + \frac{70}{3}X_{0}^{3}X_{3} + 60 X_{0}^{3}X_{4} + \frac{10}{3}X_{0}^{3}X_{5} + 5 X_{0}^{2}X_{1}^{2} + 35 X_{0}^{2}X_{2}^{2}  \\
            & + 180 X_{0}^{2}X_{2}X_{3} + 60 X_{0}^{2}X_{2}X_{4} + 30 X_{0}^{2}X_{3}^{2} + 60 X_{0}X_{2}^{3} + 60 X_{0}X_{2}^{2}X_{3} + 5 X_{2}^{4} \in \calS_4.
        \end{align*}
        Note that $Y_0 \circ G = F$. Moreover, $Z = Z_{G,X_0}$, i.e., it is evinced by the trivial GAD of $G$ given by $G = X_0^0G$. This example shows why the containment in $(d+1)$-fat points is crucial for \Cref{prop:ContainsGAD} and \Cref{GAD:inside:fat:points}. 
        In particular, we have that
        \begin{align*} g_{X_0} & = 120x_2^4 + f_{X_0} = \\
        & = 120 x_2^4 + 360 x_{2}^{3} + 120 x_{2}^{2}x_{3} + 20 x_{1}^{2} + 140 x_{2}^{2} + 360 x_{2}x_{3} + 120 x_{2}x_{4} + 120 x_{3}^{2} + 140 x_{1}\\
        &\quad  +150 x_{2}+140 x_{3}+360 x_{4}+20 x_{5}+144. \end{align*} 
       We observe that $g_{X_0}$ and $f_{X_0}$ are equal up to degree $3$, but since
       \[(y_2^2 - y_3) \contract f_{X_0} = -120 x_2^2 \neq 0,\]
       then $\Ann^{\contract}(g_{X_0}) \not\se \Ann^{\contract}(f_{X_0})$.
       \exampleend
    \end{example}


\section{Regularity of schemes evicing GADs} \label{sec:regularity}


\subsection{Apolar schemes with low multiplicities and independent supports}\label{ssec:lowmultiplicity_independent}

For a given $L \in \calS_1$, let $D_L = L^\perp \cap \calR_1$ and $D^e_L\subset\Sym^e\calR_1$ be its $e$-th symmetric power.
We also define the $\Bbbk$-vector spaces 
\[
	\calD^e_L(F) = \langle H \circ F ~:~ H \in D^e_L\rangle \subseteq \calS_{d-e}, 
\]
and given a vector space $V \subseteq \calS_m$ and $H \in \calS_l$, we write
\[ 
    H \cdot V = \{HF ~:~ F \in V\} \se \calS_{l+m}.
\] 
With the notation of the previous sections, as in \cite[Remark 3]{BJMK18:Polynomials}, we have
\[
    I(Z_{F,L})_d^\perp = \left\langle\sum_{e = 0}^d L^e \cdot \calD^e_L(F)\right\rangle \subset \calS_d.
\]

When $F=L^{d-k}G$, from the above equality and the chain rule of derivation we get 
\[ I(Z_{L^{d-k}G,L})_d^\perp = \left\langle\sum_{e = 0}^k L^{d-k+e} \cdot \calD^e_L(G)\right\rangle \subset \calS_d. \]

\begin{remark} \label{rmk:relation}
    Let $Z = \cup_{i = 1}^s Z_i \subset \Pn$ be the irreducible decomposition of a $0$-dimensional scheme.
    Then $Z$ is $h$-regular precisely when $\dim I(Z)_h^{\perp} = \deg (Z) = \sum_{i=1}^s \deg (Z_i)$, therefore there cannot be $\Bbbk$-linear relations involving generators of $I(Z_i)_h^{\perp}$ for different $i$'s.
\end{remark}

If there is a relation between different $I(Z_i)_d^{\perp}$ as in \Cref{rmk:relation}, the scheme $Z$ is not $d$-regular.
However, the following proposition shows that if such $Z$ is evinced by a GAD of $F\in \calS_d$ and is irredundant to it, such a relation cannot involve addenda appearing in that GAD.

\begin{proposition}\label{prop:externalconditions}
    Let $Z$ be the scheme evinced by the GAD $F=\sum_{i=1}^s L_i^{d-k_i}G_i \in \mathcal{S}_d$.
    If, for some $i \in \{1, \dots, s\}$, we have
\begin{equation}\label{smaller:inverse}
      L_i^{d-k_i}G_i \in \sum_{1 \leq e_i \leq k_i} L_i^{d-k_i+e_i} \cdot \mathcal{D}^{e_i}_{L_i}(G_i) + \sum_{\substack{1 \leq j \leq s \\ j \neq i}} \sum_{0 \leq e_j \leq k_j}  L_j^{d-k_j+e_j} \cdot \mathcal{D}^{e_j}_{L_j}(G_j), 
\end{equation}
    then $Z$ is redundant to $F$.
    It is intended that the first sum in \cref{smaller:inverse} is empty if $k_i=0$.
 \end{proposition}
\begin{proof}
    Without loss of generality, we may assume that in \cref{smaller:inverse} we have $i = 1$. We define a scheme $Z'$ apolar to $F$ as follows.
    
    $\bullet$ If $k_1 = 0$, by \cref{smaller:inverse}, we simplify the GAD as $F = \sum_{j=2}^s L_j^{d-k_i}G'_j$ with $G'_j \in \sum_{e=0}^{k_i}L_j^e D^e_{L_j}(G_j)$. We call $Z'$ the scheme evinced by this GAD of $F$.
    
    $\bullet$ If $k_1 > 0$, we replace $L_1^{d-k_1}G_1$ in the GAD of $F$ with the linear combination deduced from \cref{smaller:inverse}.
    In particular, there are elements $H_{j,e_j} \in \mathcal{D}_{L_j}^{e_j}$ and integers $m_j \in \bbN$ such that we can write
    \begin{equation}\label{GAD2}
        F = 
        \sum_{j=1}^s L_j^{d-k_j+m_j} \left( \sum_{m_j \leq e_j \leq k_j} L_j^{e_j-m_j} \left(H_{j,e_j} \circ G_j\right)\right).
    \end{equation} 
    Since $k_1 > 0$, then we have $m_1 \geq 1$ in \cref{GAD2}.
    The last equation is a GAD of $F$ up to deleting vanishing addenda
    and, for all the others, choosing $m_j$ such that $H_{j,m_j} \neq 0$. 
   Let $Z'$ be the scheme evinced by the new GAD in \cref{GAD2}. 
   
   By construction, $Z'$ is apolar to $F$, so it is sufficient to show that $Z' \subsetneq Z$.
   
Following the notation introduced in \Cref{sec:NAS}, let $g_j = (g_j)_{L_j} \in \calSa$ be the de-homogenization of $(G_j)_{\rm dp}$ with respect to $L_j$, and let $h_{j,e_j} \in \calRa$ be the dehomogenization of $H_{j,e_j}$ with respect to the dual linear form $L_j^*$ of $L_j$.
Since $H_{j,e_j} \in D^{e_j}_{L_j} \subset {\rm Sym}^{e_j}L_j^\perp$, then $H_{j,e_j}$ does not involve $L_j^*$, so its dehomogenization $h_{j,e_j}$ is equal to $H_{j,e_j}$.
Thus, the de-homogenization of $(H_{j,e_j}\circ G_j)_{\rm dp} = H_{j,e_j} \contract (G_j)_{\rm dp}$ with respect to $L_j$ coincides with $h_{j,e_j} \contract g_j$.
In particular, the $j$-th component of $Z'$ is defined by $\Ann^{\contract}\big(\sum_{m_j \leq e_j \leq k_j} h_{j,e_j} \contract g_j\big)$.
Since
\[
    \Ann^{\contract}\left(\sum_{m_j \leq e_j \leq k_j} h_{j,e_j} \contract g_j\right) \supseteq \Ann^{\contract}(g_j),
\]
we deduce that $Z' \subseteq Z$. We now show that this containment is proper.

In the case $k_1 = 0$, then the containment is strict because $Z'$ has no support on $[L_1]$. In the case $k_1 > 0$, since $m_1 \geq 1$, $\deg (\sum_{m_i \leq e_i \leq k_i} h_{1,e_1} \contract g_1) < \deg (g_1)$ so the socle degree of the first component of $I(Z')$ and $I(Z)$ are different, so again they must be different.
\end{proof}

\begin{proposition} \label{prop:lowmultiplicity}
    Let $s>1$ and $L_1,\ldots,L_s \in \mathcal{S}_1$ be $\Bbbk$-linearly independent forms and $Z$ be the scheme evinced by a GAD of $F \in \calS_d$ as in \cref{eq:GAD}.
    If either
    \begin{enumerate}[label=(\alph*)]
        \item\label{prop:a} $d > \max_{i \neq j}\{k_i + k_j\}$, or
        \item\label{prop:b} $d > \max_{i \neq j}\{k_i+k_j-2\}$ and $Z$ is irredundant,
    \end{enumerate}
    then $Z$ is $d$-regular. 
\end{proposition}

\begin{proof}    
    For every $1 \leq i \leq s$ we let $Z_i$ be the natural apolar scheme to $F_i=L_i^{d-k_i}G_i$ supported at $L_i$, so $Z=\cup_{i=1}^s Z_i$. 
    By \Cref{rmk:BJMRlemmas} each $Z_i$ is $d$-regular, therefore $\dim I(Z_i)_d^{\perp} = \deg (Z_i)$. 
    By \Cref{rmk:relation}, we only need to show that there are no $\Bbbk$-linear relations involving generators of $I(Z_i)_d^{\perp}$ for different $i$'s.
    If there was such a relation, there should exist $Q_i \in \sum_{e=0}^{k_i} L_i^{e}\cdot D^e_{L_i}(G_i)$, for $i = 1,\ldots,s$, such that
    \begin{equation}\label{eq:condition_prop3}
        L_i^{d-k_i}Q_i=\sum_{i\neq j}L_j^{d-k_j}Q_j.
    \end{equation}
    Since the $L_i$'s are linearly independent, up to a change of coordinates we can write the above as
    \[ X_i^{d-k_i}\tilde Q_i=\sum_{i\neq j}X_j^{d-k_j}\tilde Q_j.\]
    In case \ref{prop:a}, the hypothesis $d-k_i > k_j = \deg ( Q_j) = \deg ( \tilde Q_j)$ prevents from factoring $X_i^{d-k_i}$ out of the right-hand side of the above equation. Thus, no such relation may hold.

    In case \ref{prop:b}, since $Z$ is irredundant, by \Cref{prop:externalconditions} we may assume that any relation between the $I(Z_i)_d^{\perp}$'s does not involve any of the terms $L_i^{d-k_i}G_i$.
    Thus, \cref{eq:condition_prop3} actually leads to a relation of the form
    \[ X_i^{d-k_i+1}\tilde Q_i=\sum_{i\neq j}X_j^{d-k_j+1}\tilde Q_j. \]
    As in the previous case, the factor $X_i^{d-k_i+1}$ cannot appear on the right-hand side of the above sum due to $d > \max_{i \neq j}\{k_i+k_j\}-2$.
    
    In conclusion, in both cases, the $I(Z_i)_d^{\perp}$'s cannot intersect, so $Z$ is $d$-regular.
\end{proof}

\begin{remark}
    We note that requiring $s>1$ in \Cref{prop:lowmultiplicity} is not restrictive, as in the local case ($s=1$) \Cref{rmk:BJMRlemmas} already contains a stronger result.
\end{remark}

An immediate corollary of \Cref{prop:lowmultiplicity} is the following.


\begin{corollary}\label{cor:lowmultiplicity_1}
    Let $Z$ be the scheme evinced by the GAD $F = \sum_{i=1}^s L_i^{d-k_i}G_i \in \calS_d$, such that $L_1,\ldots,L_s$ are $\Bbbk$-linearly independent and $k_i < \frac{d}{2}$ for every $i \in \{1,\ldots,s\}$. Then $Z$ is $d$-regular.
\end{corollary}

\begin{corollary}\label{cor:lowmultiplicity_2}
    Let $Z = \bigcup_{i=1}^s Z_i$ be a $0$-dimensional scheme apolar and irredundant to $F \in \calS_d$, such that $I(Z_i) \supset \wp_{L_i}^{\lceil{\frac{d}{2}}\rceil+ 1}$ and the $L_i$ are $\Bbbk$-linearly independent.
    Then $Z$ is $d$-regular. 
\end{corollary}
\proof It follows by \Cref{GAD:inside:fat:points} together with \Cref{cor:lowmultiplicity_1}.
\endproof

\begin{remark} \label{rmk:sharp}
We notice that every requirement of \Cref{prop:lowmultiplicity} is sharp.
In fact, \Cref{ex:irredbutnotmin} shows that the inequality in \ref{prop:a} cannot be improved: if $d = \max_{i \neq j}\{k_i + k_j\}$ the scheme $Z$ may be not $d$-regular.
Similarly, the following \Cref{irregular:irredundant} shows that the inequality in \ref{prop:b} is also sharp.
Finally, \Cref{ex:perazzo} will show that the $\Bbbk$-linear independence of the supports is also needed.
\end{remark}

The following example shows that schemes that are irredundant to $F \in \calS_d$ may be not $d$-regular.

\begin{example}\label{irregular:irredundant}
    Let us consider the scheme $Z$ evinced by the GAD $F = X_0 G_1 + X_1 G_2 \in \calS_4$, where
\begin{align*}
   G_1 &= 10X_0^3 - 4X_0^2X_1 + 4X_0^2X_2 - 4X_0X_1^2 - 8X_0X_1X_2 - 3X_0X_2^2 - 8X_1^3 - 4X_2^3  \in \calS_3, \\
   G_2 &= 5X_0^3 + 9X_0X_1^2 - 5X_1^3 - 7X_1^2X_2 + 6X_1X_2^2 - X_2^3 \in \calS_3.
\end{align*}
Its defining ideal is
\[
    I(Z) = ( Y_0^3Y_1^3 - 2Y_0^3Y_2^3 + 5Y_1^3Y_2^3,
    3Y_0^2Y_1Y_2 - 2Y_0Y_2^3,
    Y_0Y_1^2Y_2,
    Y_0Y_1Y_2^2,
    Y_2^4 ),
\]
whose minimal primary decomposition is $I(Z) = I_1 \cap I_2$, where
\[
    I_1 = ( -3Y_0Y_1Y_2 + Y_1^3,
    Y_1^2Y_2,
    Y_1Y_2^2,
    Y_1^3 - 2Y_2^3 ), \quad
    I_2 = ( Y_2^4,
    Y_0^3 + 5Y_2^3,
    Y_0Y_2 ).
\]
Its Hilbert function is $\HF_Z = (1,3,6,10,11,12,12,\dots)$, hence $Z$ is not regular in degree $4 = \deg (F)$. 
\begin{claim*}
    $Z$ is irredundant to $F$.
\end{claim*}
\begin{proof}[Proof of Claim.]
The connected components of $Z$ are both contained in $4$-fat points, i.e. $I_i \supset \wp_{X_{i-1}}^4$, hence by \Cref{GAD:inside:fat:points} it is sufficient to show that the unique scheme $Y \se Z$ evinced by a GAD of $F$ of type $F = X_0^{a_0} Q_1 + X_1^{a_1} Q_2$ with $a_0,a_1 \geq 1$ is $Z$ itself.
Since in the expression of $F$ appear the monomials $-4X_0X_2^3$ and $-X_1X_2^3$, then it is easy to see that there is no such a GAD of $F$ for $a_0 > 1$ or $a_1 > 1$, therefore we assume $a_0 = a_1 = 1$. 

Since this new additive decomposition is still equal to $F$, we have
\[
    X_0(Q_1-G_1) + X_1(Q_2-G_2) = 0,
\]
hence there is $T \in \calS_2$ such that 
\[
    X_1T = Q_1-G_1, \quad X_2T = -Q_2+G_2.
\]
This means that $Y$ is evinced by a GAD of $F$ of type
\[
    F = X_0 (G_1 + X_1 T) + X_1 (G_2 - X_0 T),
\]
for some
\[
    T =  \lambda_{1} X_0^2 + \lambda_{2} X_0X_1 + \lambda_{3} X_0X_2 + \lambda_{4} X_1^2 + \lambda_{5} X_1X_2 + \lambda_{6} X_2^2 \in \calS_2.
\]
If $Y = Y_1 \cup Y_2 \se Z$, then we have
\[ I_1 \se I(Y_1) = I(Z_{X_0 (G_1 + X_1 T),X_0}) \se \Ann\big(X_0 (G_1 + X_1 T) \big), \]
which implies
\[
    \begin{cases}
        0 = (-3Y_0Y_1Y_2 + Y_1^3) \circ \big(X_0(G_1 + X_1 T)\big) = 6(-\lambda_3 + \lambda_4)X_0 - 6\lambda_5X_1 - 6\lambda_6X_2, \\
        0 = (Y_1^2Y_2) \circ \big(X_0(G_1 + X_1 T)\big) = 2\lambda_5X_0, \\
        0 = (Y_1Y_2^2) \circ \big(X_0(G_1 + X_1 T)\big) = 2\lambda_6X_0, \\
        0 = (Y_1^3 - 2Y_2^3) \circ \big(X_0(G_1 + X_1 T)\big) = 6\lambda_4X_0.
    \end{cases} 
\]
Similarly, from $I_2 \se \Ann\big(X_1 (G_2 + X_0 T)\big)$ we obtain
\[
    \begin{cases}
        0 = (Y_2^4) \circ \big(X_1(G_2 + X_0 T)\big) = 0, \\
        0 = (Y_0^3 + 5Y_2^3) \circ \big(X_1(G_2 + X_0 T)\big) = -6\lambda_1X_1, \\
        0 = (Y_0Y_2) \circ \big(X_1(G_2 + X_0 T)\big) = -2\lambda_3X_0X_1 - \lambda_5X_1^2 - 2\lambda_6X_1X_2.
    \end{cases} 
\]
The above systems imply
\[
    \lambda_1 = \lambda_3 = \lambda_4 = \lambda_5 = \lambda_6 = 0,
\]
thus we conclude that $T = \lambda_2 X_0X_1$.
We computationally verify that the scheme evinced by the GAD
\[
    X_0 (G_1 + \lambda_2 X_0X_1^2) + X_1 (G_2 - \lambda_2 X_0^2X_1)
\]
is independent on $\lambda_2 \in \Bbbk$, and it is always equal to $I(Z)$.
Therefore, we conclude that $Y = Z$, i.e. $Z$ is irredundant.
\end{proof}

The proof of the above claim shows an effective way for establishing irredundancy to $F$ by symbolically testing its GADs.
\exampleend
\end{example}

\subsection{Tangential decompositions} \label{sec:tangential}

In this section, we prove that if a minimal apolar scheme to $F \in \calS_d$ is a union of simple points and {\it $2$-jets} (i.e. local $0$-dimensional schemes of length $2$), then it is $d$-regular.
Such schemes are evinced by GADs as in \cref{tg:decomp}, which are called \emph{tangential decompositions} due to their relation with secant varieties of tangential varieties of Veronese varieties \cite{BT20:Waring, CGG}.

\begin{proposition}\label{prop:2jets}
    Let $Z = Z_1 \cup \ldots \cup Z_r$ such that $\len(Z_i) \leq 2$ for every $i \in \{1, \dots, r\}$.
    If $Z$ is of minimal length among the apolar schemes to $F \in \calS_d$ of this type, then $Z$ is $d$-regular.
\end{proposition}

\begin{proof}
By \Cref{GAD:inside:fat:points}, $Z$ is evinced by a GAD of $F$ of type
\begin{equation}\label{tg:decomp}
    F=\sum_{i=1}^s L_i^{d-1}G_i + \sum_{i=s+1}^r L_i^d
    \end{equation}
for some $0 \leq s \leq r$ and $L_i,G_i \in \calS_1$.
Moreover, we have
\[ I(Z)^{\perp}_d = \langle L_i^d, L_j^{d-1}G_j \rangle_{ \substack{1 \leq i \leq r\\ 1 \leq j \leq s}}. \] 
Since $\len(Z)$ is $r+s$, which is also equal to the number of generators of $I(Z)^{\perp}_d$, in order to prove that $Z$ is $d$-regular it is sufficient to show that all those generators are $\Bbbk$-linearly independent.
We prove that if there is a linear relation between the $L_i^d$'s and the $L_j^{d-1}G_j$'s as above, then we can explicitly produce a scheme whose connected components have length at most $2$ and with a length smaller than $\len(Z)$, contradicting its minimality.

When such a relation involves an addenda appearing in the above GAD of $F$, then $Z$ is redundant by \Cref{prop:externalconditions}, which would contradict the minimality assumption of $Z$.
Thus, if there is a linear relation among the generators of $I(Z)^{\perp}_d$, it can only involve $\{L_i^d\}_{1 \leq i \leq s}$, which are the unique generators of $I(Z)^{\perp}_d$ that do not appear in the GAD of \cref{tg:decomp}. Therefore if we show that $L_1^d, \ldots, L_s^d$ are linearly independent we are done.
We will prove a stronger fact, namely that $L_1^{d-1}, \ldots , L_s^{d-1}$ are linearly independent.
Suppose by contradiction that $L_1^{d-1} = \sum_{i=2}^s \lambda_i L_i^{d-1}$ for some $\lambda_i\in \Bbbk$. By substituting this relation in \cref{tg:decomp}, we get
\[
    F=\sum_{i=2}^s L_i^{d-1}(G_i+\lambda_i G_1)+ \sum_{i=s+1}^r L_i^d.
\]
The scheme $Z'$ evinced by this new GAD of $F$ is still apolar to it and has components of length at most $2$, but its length is at most $s+r-2= \len (Z)-2 < \len (Z)$. 
\end{proof}

Notice that in the proof of \Cref{prop:2jets} we have employed the length-minimality of the scheme $Z$ apolar to $F$.
Indeed, the irredundancy of an apolar scheme of $2$-jets is not sufficient to guarantee the regularity in the degree of $F$, as shown in the following example.

\begin{example}\label{ex:perazzo}
    Let $Z$ be the scheme evinced by the following GAD of $F \in \calS_3$:
    \[ F = X_0^2X_2 + X_1^2X_3 + (X_0+X_1)^2X_4 + (X_0-X_1)^2(X_2-3X_3-2X_4)+ (X_0+2X_1)^2(X_2+X_3+X_4). \]
    
    It is easy to check that $F$ is written in essential variables \cite{IaKa:book, carlini}, and that $Z$ is the union of five $2$-jets $Z_1, \ldots , Z_5$ supported on points $[L_1], \dots, [L_5] \in \Pn$ of the rational normal cubic. 
    Its Hilbert function is $\HF_Z = (1,5,8,9,10,10, \ldots)$, therefore 
    $Z$ is not regular in degree $3 = \deg(F)$.

    However, $Z$ is irredundant: any proper subscheme of $Z=\cup_{i=1}^5Z_i$ has to be contained in one of the following, for $i \in \{1, \dots, 5\}$:
    \[ Y_i = [L_i] \cup \bigcup_{j \neq i} Z_j. \]
    We computationally verify that for every $i$ we have $I(Y_i) \subsetneq \Ann(F)$, therefore no proper subscheme of $Z$ is apolar to $F$. 


    We now verify that the strategy of \Cref{prop:2jets} produces an apolar scheme that is shorter than $Z$, but not contained in it. Substituting the relation
    \[ (X_0-X_1)^2 = 2X_0^2+2X_1^2-(X_0+X_1)^2 \]
    we obtain the new GAD
    \begin{align*}
        F =&\ X_0^2 (3 X_2-6 X_3-4 X_4) + X_1^2 (2 X_2-5 X_3-4 X_4) + (X_0+X_1)^2 (-X_2+3 X_3+3 X_4) \\
        &\ + (X_0+2X_1)^2 (X_2 + X_3 + X_4).
    \end{align*} 
    The scheme evinced by this GAD has length $8$ but is not contained in $Z$.
    We can repeat the procedure with the relation
    \[ (X_0+2X_1)^2 = 2(X_0+X_1)^2-X_0^2+2X_1^2, \]
    which leads us to another GAD
    \[ F = X_0^2 (2 X_2-7 X_3-5 X_4) + X_1^2 (4 X_2-3 X_3-2 X_4) + (X_0+X_1)^2 (X_2+5 X_3+5 X_4). \]
    The scheme evinced the last GAD is minimal among the apolar schemes to $F$: as it has length $6$ and, up to a change of variables, $F$ is the Perazzo cubic \cite{perazzo} which has cactus rank 6 (see eg. \cite[Example 2.8]{BBM14:comparison}, \cite[Section 4]{bb:differences}).
    This can also be directly verified with \cite[Algorithm~3]{BT20:Waring}.
    \exampleend
\end{example}

\begin{remark} \label{rem:Tanleqd}
    As a corollary of \Cref{prop:2jets}, a (minimal) tangential decomposition can be reached by \cite[Algorithm 2]{BT20:Waring} by only considering bases of degree up to $d$, as in the Waring case.
    In fact, \Cref{prop:2jets} provides an analogous of \cite[Proposition 3.4]{BT20:Waring} for the tangential case, which positively answers \cite[Remark 5.4]{BT20:Waring}, hence only bases in $\mathcal{B}_d$ can be considered. \\
    If we could obtain a similar result to \Cref{prop:2jets} for schemes evincing the cactus rank (\Cref{q:main}), we would have the analogous bound for \cite[Algorithm 3]{BT20:Waring}.
\end{remark}

\subsection{Apolar schemes with low length}

\begin{proposition} \label{prop:shortschemes}
Let $Z \subset \Pn$ be a $0$-dimensional scheme apolar and irredundant to $F \in \calS_d$. If $\len(Z) \leq 2d+1$, then $Z$ is $d$-regular.
\end{proposition}
\proof By contradiction, let us assume that $Z$ is not $d$-regular. Then, by \cite[Lemma~34]{BGI11}, there exists a line $L$ such that $\len (Z \cap L) \geq d+2$. Let $\Res_L(Z)$ be the residual scheme of $Z$ with respect to $L$ defined by the colon ideal $\big( I(Z) : (L) \big)$.
Since 
\[ \len(Z \cap L) + \len\big( \Res_L(Z) \big) = \len(Z) \leq 2d+1, \]
hence we have $\len\big( \Res_L(Z) \big) \leq d-1 < \len(Z \cap L)$.
Given the irreducible decomposition $Z = Z_1 + \cdots + Z_s$, there exists a component $Z_i$ such that the schematic intersection $Z_i \cap L$ satisfies 
\begin{equation} \label{eq:lenlen}
    \len(Z_i \cap L) > \len(\Res_L(Z_i)).
\end{equation}
Without loss of generality, we may assume that $i = 1$, $I(Z_1) \se \wp_{X_0}$ and $I(L) = ( X_1, \dots, X_n )$.
Let $H$ be the orthogonal hyperplane to $X_0$, i.e. $I(H) = ( X_0 )$, and let $m = \len(Z_1 \cap L)$.
We consider the scheme $Z'$ defined by
\[ I(Z') = I\left(Z_1 \cap (m-1)H \right) \cap I(Z_2) \cap \dots \cap I(Z_s). \]
It is clear that $Z' \subsetneq Z$, hence to get the desired contradiction it is sufficient to show that $Z'$ is apolar to $F$, which follows directly from the following fact by the Apolarity Lemma (\Cref{lemma:apolarity}).

\begin{claim}\label{claim:lowlength}
    $I(Z)^\perp_d = I(Z')^\perp_d$.
\end{claim}

\emph{Proof of \Cref{claim:lowlength}.}

Since $m > \len\big(\Res(Z_1)\big)$ by \cref{eq:lenlen}, we have $( X_0^{m-1} ) \cap ( X_1,\ldots,X_n ) \se I(Z_1)$, hence
\[ I(Z_1) = \big( I(Z_1) + ( X_0^{m-1} ) \big) \cap \big( I(Z_1)+ ( X_1,\ldots,X_m ) \big). \]
$\bullet$ We prove that $I(Z_1) + ( X_0^{m-1} )$ equals the saturated ideal $I\big(Z_1 \cap (m-1)H \big)$.

There are obvious ideal inclusions:
\begin{equation} \label{eq:inclusions}
I(Z_1) \se I(Z_1) + ( X_0^{m-1} ) \se I\big(Z_1 \cap (m-1)H\big).
\end{equation}
It is enough to show that the last two ideals have the same Hilbert function.
Since $Z_1 \cap (m-1)H$ has colength $1$ inside $Z_1$ and their homogeneous defining ideals agree up to degree $m-2$, we deduce that 
\[\HF_{Z_1 \cap (m-1)H}(i) = \begin{cases}
        \HF_{Z_1}(i) & \text{for $i \leq m-2$,} \\
        \HF_{Z_1}(i)-1  & \text{for $i \geq m-1$.}
    \end{cases}\]
By \cref{eq:inclusions} the Hilbert function $\HF_*$ of $\calS/\big( I(Z_1) + ( X_0^{m-1} ) \big)$ is squeezed: $\HF_{Z_1 \cap (m-1)H} \leq \HF_* \leq \HF_{Z_1}$.
However, for every $k \geq m-1$ we have $X_0^{k} \in \left( I(Z_1) + ( X_0^{m-1} ) \right) \setminus I(Z_1)$, thus $\HF_*(k) < \HF_{Z_1}(k)$ for every $k \geq m-1$.
This implies that $\HF_*$ completely agrees with $\HF_{Z_1 \cap (m-1)H}$.

$\bullet$ For every $i \in \{2,\ldots,s\}$, we trivially have
\[
          I(Z_i) = I(Z_i) \cap \big( I(Z_i) + ( X_1,\ldots,X_n ) \big).
\]
Hence, we can write:
\begin{align*}
    I(Z) &= I\big( Z_1 \cap (m-1)H \big) \cap \big( I(Z_1)+ ( X_1,\ldots,X_m ) \big) \cap \bigcap_{i = 2}^s \big( I(Z_i) + ( X_1,\ldots,X_n ) \big) \cap I(Z_i) \\
    &= I(Z') \cap \left( \bigcap_{i=1}^s I(Z_i) + ( X_1,\ldots,X_n ) \right) = I(Z') \cap I(Z \cap L).
\end{align*}
$\bullet$ From the non-degeneracy of the apolar action we get
    \[
    I(Z)_d^\perp = [I(Z') \cap I(Z\cap L)]_d^\perp = I(Z')_d^\perp + I(Z \cap L)_d^\perp \]
    but $I(Z \cap L)_d = I(Z' \cap L)_d$ because they define schemes of length $d+1$ on the same normal curve $\nu_d(L) \subset \bbP^d$. Thus, we conclude
    \[
    I(Z)_d^\perp = I(Z')_d^\perp + I(Z' \cap L)_d^\perp =  I(Z')_d^\perp,
    \]
    which proves the claim and then concludes the proof. \endproof

We notice that \Cref{prop:shortschemes} provides a good criterion for proving that the minimal apolar schemes to a \emph{given} $F \in \calS_d$ is $d$-regular, by exhibiting at least one scheme $Z$ apolar to $F$ and of length not bigger than $2d+1$.

\begin{example}
 Let $F \in \calS_4$ be the polynomial considered in \Cref{irregular:irredundant}.
 We consider another GAD $F = X_0 \tilde G_1 + X_1 \tilde G_2$, where
 \begin{align*}
   \tilde G_1 &= 10 X_0^3 + X_0^2 X_1 + 4 X_0^2 X_2 - 4 X_0 X_1^2 - 8 X_0 X_1 X_2 - 3 X_0 X_2^2 - 4 X_2^3 \in \calS_3, \\
   \tilde G_2 &= X_0 X_1^2 - 5 X_1^3 - 7 X_1^2 X_2 + 6 X_1 X_2^2 - X_2^3 \in \calS_3.
\end{align*}
 This GAD evinces the scheme $\tilde Z$ defined by 
 \[ I(\tilde Z) = \left( Y_0^2 Y_1 Y_2 - \frac{2}{3} Y_0 Y_2^3,
    Y_0 Y_1 Y_2^2,
    Y_2^4,
    Y_0 Y_1^2 - \frac{5}{2} Y_0 Y_1 Y_2 + Y_2^3 \right). \]
 Its Hilbert function is $\HF_Z = (1,3,6,9,9,\dots)$. 
 Since $\len(Z) = 9 \leq 2 \cdot 4 + 1$, by \Cref{prop:shortschemes} we can guarantee that minimal schemes apolar to such a $F$ are $4$-regular, even without computing them.
\end{example}


\section{Conclusion} \label{sec:Conclusion}


In the present work, we investigated the $d$-regularity of certain families of schemes apolar to $F \in \calS_d$.
In all the examples we presented, the schemes of minimal lengths were $d$-regular, so it is natural to ask whether this is always the case.
\begin{question}\label{q:main}
    Let $F \in \calS_d$ and $Z$ be a $0$-dimensional scheme evincing its cactus rank. Is $Z$ $d$-regular?
\end{question}

Actually, a careful reader would have noticed that none of the examples we considered really required to reach degree $d$ for regularity, instead $d-1$ was sufficient.
The most general way to phrase our question is the following.

\begin{question}\label{q:main+}
    Is there a sharp uniform upper bound (independent of the number of variables) for the regularity of schemes evincing the cactus rank of a degree-$d$ form?
\end{question}

To the best of our knowledge, we do not know the answer to \Cref{q:main,} and \Cref{q:main+}.
We believe that our results and examples could be useful in either direction.
Our positive results restrict the identikit of a possible example providing a negative answer to \Cref{q:main} to have some component of high multiplicity.
On the other side, when trying to prove a positive answer to \Cref{q:main,} or \Cref{q:main+}, \Cref{ex:irredbutnotmin} shows that we really need to use the {\it global} assumption of minimality in terms of the cactus rank, and that cannot be relaxed with the {\it local} condition of minimality by inclusion.

\bibliographystyle{alpha}
\bibliography{new_bib.bib}

Declarations of interest : none.

\end{document}